\documentclass{amsart}

\usepackage{amsmath}

\usepackage{a4wide, mathrsfs, stmaryrd}

\usepackage{amssymb, amsmath}

\usepackage[all,arrow,matrix,curve]{xy}

\usepackage{fancyheadings}

\newdir{ >}{{}*!/-5pt/\dir{>}}
\newdir{ -->}{{}*!/-10pt/\dir{>}}

\newcommand{\nc}{\newcommand}
\nc{\al}{\alpha}
\nc{\bt}{\beta}
\nc{\Lb}{\Lambda}
\nc{\gm}{\gamma}
\nc{\Gm}{\Gamma}
\nc{\vf}{\varphi}
\nc{\ve}{\varepsilon}
\nc{\sg}{\sigma}
\nc{\om}{\omega}
\nc{\tht}{\theta}
\nc{\Tht}{\Theta}
\nc{\hsm}{\,{-\hskip-3.2mm{\small>}}\,}
\nc{\hsn}{{\langle-\hskip-2mm{\small>}\rangle}}
\nc{\lb}{\lambda}
\nc{\dl}{\delta}
\nc{\vkp}{\varkappa}
\nc{\all}{\allowdisplaybreaks}
\nc{\wt}{\widetilde}
\nc{\os}{\overset}
\nc{\ol}{\overline}
\nc{\lr}{\longrightarrow}
\nc{\tm}{\times}
\nc{\cT}{\textrm{\!\textcalligra T\;\;}}
\nc{\cC}{\textrm{\!\textcalligra C\;\;}}
\nc{\ha}{\hookrightarrow}
\nc{\cE}{\mathcal{E}}
\nc{\rt}{\rightarrowtail}
\nc{\ta}{\twoheadrightarrow}
\nc{\bR}{\Bbb{R}}
\nc{\bP}{\Bbb{P}}
\nc{\bF}{\Bbb{F}}
\nc{\bQ}{\Bbb{Q}}
\nc{\btd}{\bigtriangledown}
\nc{\btu}{\bigtriangleup}
\nc{\pa}{\partial}
\nc{\sbs}{\subset}
\nc{\ul}{\underline}

\makeatletter
\@addtoreset{equation}{section}
\makeatother

\numberwithin{equation}{section}
\newtheorem{theo}[equation]{Theorem}
\newtheorem{lem}[equation]{Lemma}
\newtheorem{prop}[equation]{Proposition}
\newtheorem{cor}[equation]{Corollary}
\theoremstyle{definition}
\newtheorem{defi}[equation]{Definition}

\newtheorem{remk}[equation]{Remark}
\newtheorem{exmp}[equation]{Example}

\newtheorem{paragr}[equation]{}

\swapnumbers
\newtheorem{numbered paragraph}[equation]{}

\newcommand{\Ker}{\operatorname{Ker}}

\advance\baselineskip by.3\baselineskip

\thanks{This work was supported by Shota Rustaveli National Science Foundation Grant DI/18/5-113/13.}

\begin{document}

\def\N{{\mathbb N}}
\def\Z{{\mathbb Z}}

\title[]
{Cohomology monoids of monoids with coefficients in semimodules II}

\author{Alex Patchkoria}



\maketitle

\begin{abstract}
We relate the old and new cohomology monoids of an arbitrary monoid $M$ with coefficients in semimodules over $M$, introduced in the author's previous papers, to monoid and group extensions. More precisely, the old and new second cohomology monoids describe Schreier extensions of semimodules by monoids,  and the new third cohomology monoid is related to a certain group extension problem.
\end{abstract}

\setcounter{section}{0}

\vspace{0.5cm}

\vskip+7mm

\section{Introduction}
In \cite{P1,P2}, in order to give a cohomological description of Schreier extensions of semimodules by monoids, we introduced cohomology monoids of an arbitrary monoid $M$ with coefficients in semimodules over $M$. Later, in \cite{HLS}, developing the basic idea of  Sweedler's  two-cocycles of \cite{S}, Haile, Larson and Sweedler introduced the same cohomology monoids (in the case where $M$ is a group) as a generalization of the usual Amitsur and Galois cohomology groups. Next, in \cite{P3,P4,P5} we developed a version of homological algebra for semimodules which has many advantages over the old one, used in \cite{P1,P2,HLS}, from the point of view of applications in algebraic topology (see \cite{P3,P4,P6} for details). The new version of homological algebra for semimodules gave rise to new cohomology monoids of an arbitrary monoid $M$ with coefficients in semimodules over $M$. In \cite{P6}, we introduced these cohomology monoids and showed that they are more adequate for actual computations. In particular, we calculated them in the case where $M$ is a finite cyclic group by using the technique of free resolutions.

Let $M$ be a monoid and $A$ a (left) $M$-semimodule, and let $\mathscr{H}^n(M,A)$ and $H^n(M,A)$ denote, respectively,  the old and new cohomology monoids with coeffcients in the $M$-semimodule A. In the present paper, we relate $H^2(M, A)$ to Schreier extensions of $A$ by $M$. A cohomological classification of Schreier extensions of $A$ by $M$ via $\mathscr{H}^2(M,A)$, which is part of the author's unpublished thesis \cite{PhD}, is also given. The paper is completed with an application of $H^3(\Pi,A)$, where $\Pi$ is a group and $A$ a $\Pi$-semimodule, to a certain group extension problem.

\section{Preliminaries}
A semiring $\Lb=(\Lb,+\,,0,\cdot\,,1)$ is an algebraic structure
in which $(\Lb,+\,,0)$   is an abelian monoid, $(\Lb,\cdot\,,1)$ a
monoid, and
\begin{align*}
\lb\cdot(\lb'+\lb'')&=\lb\cdot\lb'+\lb\cdot\lb'',\\
(\lb'+\lb'')\cdot\lb&=\lb'\cdot\lb+\lb''\cdot\lb,\\
\lb\cdot 0=0\cdot\lb&=0
\end{align*}
for all $\lb,\lb',\lb''\in\Lb$ (see e.g. \cite{Gol}).
\vskip+0.8mm
Let $\Lb$ be a semiring. An abelian monoid $A=(A,+\,,0)$ together
with a map $\Lb\tm A\lr A$, written as $(\lb,a)\mapsto\lb a$, is
called a (left) $\Lb$-semimodule if
\begin{align*}
\lb(a+a')&=\lb a+\lb a',\\
(\lb+\lb')a&=\lb a+\lb' a,\\
(\lb\cdot \lb')a&=\lb(\lb' a),\\
1a=a,&\quad 0a=0
\end{align*}
for all $\lb,\lb'\in \Lb$ and $a,a'\in A$. It immediately follows that $\lb 0=0$ for any $\lb\in\Lb$.
\vskip+0.8mm

A map $f:A\lr B$ between $\Lb$-semimodules $A$ and $B$ is called
a $\Lb$-homomorphism if $f(a+a')=f(a)+f(a')$ and $f(\lb a)=\lb
f(a)$ for all $a,a'\in A$ and  $\lb\in\Lb$. It is obvious that
any $\Lb$-homomorphism carries 0 into 0.
\vskip+0.8mm
 A $\Lb$-subsemimodule $A$ of a   $\Lb$-semimodule $B$ is a subsemigroup of $(B,+)$ such that
$\lb a\in A$  for all    $a\in A$ and $\lb\in\Lb$. Clearly, $0\in A$.
\vskip+0.8mm
Let $\N$ be the semiring of nonnegative integers. An $\N$-semimodule $A$ is simply an abelian
monoid, an $\N$-homomorphism $f:A\lr B$ is just a homomorphism of abelian monoids, and
$C$ is an $\N$-subsemimodule of an $\N$-semimodule $D$ if and only if $C$ is a submonoid of
the monoid $(D,+\,,0)$.
\vskip+0.8mm
An equivalence relation $\rho$ on a $\Lb$-semimodule $A$ is said to be a congruence if it preserves the $\Lb$-semimodule structure (i.e., $a\rho a'$ and $b\rho b'$ in $A$ imply $(a+b)\rho (a'+b')$ and $(\lambda a) \rho (\lambda a')$ for all $\lambda\in\Lambda$). In this case, the quotient set $A/\rho$ is in fact a $\Lb$-semimodule
($cl(a)+cl(a')=cl(a+a'), \lambda cl(a)=cl(\lambda a)$ for all $a,a'\in A$ and $\lambda\in\Lambda$), called the quotient $\Lb$-semimodule of $A$ by $\rho$.
\vskip+0.8mm
Next recall that the group completion of an abelian monoid $S$ can
be constructed in the following way.
Define an equivalence relation $\sim$ on $S\tm S$ as follows:
$$(u,v)\sim(x,y)\Leftrightarrow u+y+z=v+x+z\quad\text{for some}\quad z\in S.$$
Let $[u,v]$ denote the equivalence class of $(u,v)$. The quotient set
$(S\tm S)/\!\!\sim$ with the addition $[x_1,y_1]+[x_2,y_2]=[x_1+x_2,y_1+y_2]$ is an abelian group
$(0=[x,x],\;-[x,y]=[y,x])$. This group, denoted by $K(S)$, is the group completion of $S$,
and $k_S:S\lr K(S)$ defined by $k_S(x)=[x,0]$ is the canonical homomorphism. If $S$ is
a semiring, then the multiplication $[x_1,y_1]\cdot[x_2,y_2]=[x_1x_2+y_1y_2,x_1y_2+y_1x_2]$
converts  $K(S)$ into the ring completion of the semiring $S$, and $k_S$ into the canonical
semiring homomorphism. Now assume that $A$ is a $\Lb$-semimodule. Then $K(A,+,0)$ with the
multiplication $[\lb_1,\lb_2][a_1,a_2]=[\lb_1a_1+\lb_2a_2,\lb_1a_2+\lb_2a_1]$,
$\lb_1,\lb_2\in\Lb$, $a_1,a_2\in A$,  becomes a $K(\Lb)$-module. This $K(\Lb)$-module,
denoted by $K(A)$, is the $K(\Lb)$-module completion of the $\Lb$-semimodule $A$, and
$k_A=k_{(A,+,0)}$ is the canonical $\Lb$-homomorphism. Clearly,
$K(A)$ is an additive functor: for any homomorphism
$f:A\lr B$ of $\Lb$-semimodules, $K(f):K(A)\lr K(B)$ defined by
$K(f)([a_1,a_2])=[f(a_1),f(a_2)]$ is a $K(\Lb)$-homomorphism.
\vskip+0.8mm

A $\Lb$-semimodule $A$ is said to be cancellative if whenever  $a+a'=a+a'',\;a,a',a''\in A$,
one has $a'=a''$. Obviously, $A$ is cancellative if and only if the canonical
$\Lb$-homomorphism $k_A:A\lr K(A)$ is injective. Consequently, for a cancellative $\Lb$-semimodule $A,$ one may assume that $A$ is a $\Lb$-subsemimodule of $K(A)$, and that each element $b$ of $K(A)$ is a difference of two elements from $A$, i.e., $b=a_1-a_2$, where $a_1, a_2 \in A.$
\vskip+0.8mm

A $\Lb$-semimodule $A$ is called a $\Lb$-module if $(A,+,0)$ is an
abelian group. One can easily see that  $A$ is a $\Lb$-module  if
and only if $A$ is a $K(\Lb)$-module. Hence, if $A$ is a
$\Lb$-module, then $K(A)=A$ and $k_A=1_A$. For a $\Lb$-semimodule $A$, by $U(A)$ we denote the maximal $\Lb$-submodule of $A$, i.e., $$U(A)=\{a \in A \;\vert \; a+a'=0\; \;\text{for some}\; \; a' \in A\}.$$
\vskip+0.8mm

Let $M$ be a multiplicatively written monoid. The free abelian monoid $\N[M]$ generated by the elements $x\in M$ consists of the formal
sums $\sum\limits_{x\in M}n_xx$, where $n_x\in\N$ and all but a finite number of the $n_x$ are zero.  The product in $M$ induces a product
$$\sum_{x\in M}n_xx\cdot\sum_{y\in M}n'_yy=\sum_{x,y\in M}(n_xn'_y)xy$$ of two such sums, and makes $\N [M]$ a semiring, the
monoid semiring of $M$ with nonnegative integer coefficients. Semimodules over $\N[M]$ are called $M$-semimodules.

\section{Cohomology monoids}
The main purpose of this section is to recall from \cite{P6} the notion of cohomology monoids with coefficients in semimodules.

\begin{defi}[\cite{P3}]  We say that a sequence of $\Lb$-semimodules and
$\Lb$-homomorphisms
$$\xymatrix{X:\cdots\ar@<0.45ex>[r]\ar@<-0.45ex>[r]&X_{n+1}
\ar@<0.45ex>[r]^-{\pa_{n+1}^+}\ar@<-0.45ex>[r]_-{\pa_{n+1}^-}&
X_n\ar@<0.45ex>[r]^-{\pa_n^+}\ar@<-0.45ex>[r]_-{\pa_n^-}&
X_{n-1}\ar@<0.45ex>[r] \ar@<-0.45ex>[r] & \cdots},\quad n \in \Z,$$
written $X=\{X_n,\pa_n^+,\pa_n^-\}$ for short, is a {\em chain complex} if
$$\pa_n^+\,\pa_{n+1}^++\pa_n^-\,\pa_{n+1}^-=\pa_n^+\,\pa_{n+1}^-+\pa_n^-\,
\pa_{n+1}^+$$
for each integer $n$. For every chain complex $X$, we define the $\Lb$-semi\-mo\-dule
$$Z_n(X)=\big\{x\in X_n\;|\;\pa_n^+(x)=\pa_n^-(x)\big\},$$
the $n$-{\em cycles}, and the $n$-{\em th} {\em homology} $\Lb$-{\em semimodule}
$$H_n(X)=Z_n(X)/\rho_n(X),$$
where $\rho_n(X)$ is a congruence on $Z_n(X)$ defined as follows:
\begin{align*}
x\;\rho_n(X)\;y\;\;\Leftrightarrow \;\; &x+\pa_{n+1}^+(u)+\pa_{n+1}^-(v)=y+\pa_{n+1}^+(v)+
\pa_{n+1}^-(u)\\
&\;\text{for some}\quad u,v \quad\text{in}\quad X_{n+1}.
\end{align*}
The $\Lb$-homomorphisms $\pa_n^+,\pa_n^-$ are called \emph{differentials} of the chain complex $X$. \end{defi}

\vskip+2mm
\begin{defi}[\cite{P3}] Let $X=\{X_n,\pa_n^+,\pa_n^-\}$ and
$X'=\{X'_n,{\pa_n^{\,'}}^+,{\pa_n^{\,'}}^-\}$ be chain complexes of
$\Lb$-semimodules. We say that a sequence $f=\{f_n\}$ of $\Lb$-homomorphisms $f_n:X_n\lr
X'_n$ is a $\pm$-{\em morphism} from $X$ to $X'$  if
$$f_{n-1}\pa_n^+={\pa_n^{\,'}}^+f_n\quad\text{and}\quad f_{n-1}\pa_n^-={\pa_n^{\,'}}^-f_n
\quad\text{for all}\quad n. $$ \end{defi}
If  $f=\{f_n\}:X\lr X'$ is a $\pm$-morphism of chain complexes, then
$f_n(Z_n(X))\sbs Z_n(X')$, and the map
$$H_n(f):H_n(X)\lr H_n(X'),\;\;H_n(f)(cl(x))=cl(f_n(x)),$$
is a homomorphism of $\Lb$-semimodules. Thus $H_n$ is a covariant additive functor from the
category of chain complexes and their $\pm$-morphisms to the category of $\Lb$-semimodules.
\vskip+4mm

\begin{paragr}[\cite{P3}]\label{ipl}{\em A cochain complex} is a sequence of $\Lb$-semimodules and $\Lb$-homomorphisms
$$\xymatrix{Y=\{Y^n,\delta^n_+,\delta^n_-\}:\;\;\; \cdots \ar@<0.6ex>[r] \ar@<-0.6ex>[r] & Y^{n-1} \ar@<0.6ex>[r]^-{\delta^{n-1}_{+}} \ar@<-0.6ex>[r]_-{\delta^{n-1}_{-}} & Y^n \ar@<0.6ex>[r]^-{\delta^{n}_{+}} \ar@<-0.6ex>[r]_-{\delta^{n}_{-}}  & Y^{n+1} \ar@<0.6ex>[r] \ar@<-0.6ex>[r] & \cdots }$$
with
 $$\dl_+^n\dl_+^{n-1}+\dl_-^n\dl_-^{n-1}=\dl_+^n\dl_-^{n-1}+\dl_-^n\dl_+^{n-1}$$
for all $n$. One obviously defines the $n$-{\em cocycles} of $Y$, the {\em $n$-th cohomology $\Lb$-semimodule} $H^n(Y)$, a $\pm$-morphism $g:Y \lr Y'$ of cochain
complexes and the induced $\Lb$-homomorphism $H^n(g):H^n(Y) \lr H^n(Y')$. Clearly, $H^n$ is a covariant additive functor on the category of cochain complexes and their $\pm$-morphisms with values in the category of $\Lb$-semimodules.\end{paragr}

\begin{paragr}[\cite{P3}]\label{ipli} A sequence $G=\{G^n,d^n_+,d^n_-\}$ of $\Lb$-modules and $\Lb$-homomorphisms is a cochain complex if and only if
$$\xymatrix{\{G^n,d^n_+-d^n_-\}:\;\;\;\;\;\cdots\ar[r]&G^n\ar[r]^-{d^n_+-d^n_-}&G^{n+1}\ar[r]&\cdots}$$
is an ordinary cochain complex of $\Lb$-modules. Obviously, for any cochain complex $ G=\{G^n,d^n_+,d^n_-\}$ of $\Lb$-modules, $H^*(G)$ coincides with the usual cohomology $H^*(\{G^n,d^n_+-d^n_-\})$. \end{paragr}

\begin{paragr}[\cite{P4}]\label{iplik}If $X=\{X^n,\delta^n_+,\delta^n_-\}$ is a cochain complex of $\Lb$-semimodules, then
$$\xymatrix{K(X) \colon \cdots \ar[r] & K(X^{n-1}) \ar[rrr]^-{K(\delta^{n-1}_+)- K(\delta^{n-1}_-)} & & & K(X^n) \ar[rr]^-{K(\delta^n_+)- K(\delta^n_-)} & & K(X^{n+1}) \ar[r] & \cdots}$$
is an ordinary cochain complex of $K(\Lb)$-modules (i.e., $\Lb$-modules) by \ref{ipli}. When each $X^n$ is cancellative, then the converse is also true.
Further, for any cochain complex $X=\{X^n,\delta^n_+,\delta^n_-\}$ of $\Lb$-semimodules, the canonical $\pm$-morphism $k_X=\{k_{X^n}:X^n\lr K(X^n)\}$ from $X$ to the cochain complex $\{K(X^n),K(\delta^n_+),K(\delta^n_-)\}$ induces the $\Lb$-homomorphisms $H^n(k_X):H^n(X)\lr H^n(K(X))$, $H^n(k_X)(cl(x))=cl(k_{X^n}(x))=cl[x,0]$.
If $X$ is a cochain complex of cancellative $\Lb$-semimodules, then $H^n(k_X)$ is injective and therefore $H^n(X)$ is a cancellative $\Lb$-semimodule for each $n$.
\end{paragr}
\vskip+1.5mm
\begin{paragr}\label{ipliki} In \cite{P2} and \cite{HLS}, another construction of cohomology  of a cochain complex $X=\{X^n,\delta^n_+,\delta^n_-\}$ is used, namely, $$\mathscr{H}^n(X)=Z^n(X)/\widetilde{\rho}\;^n(X),\quad n \in \Z,$$ where $Z^n(X)=\big\{x\in X^n\;|\;\delta^n_+(x)=\delta^n_-(x)\big\}$ and a congruence $\widetilde{\rho}\;^n(X)$ on the $\Lb$-semimodule $Z^n(X)$ is defined by
\begin{align*} x\;\widetilde{\rho}\;^n(X)\;y\;\;\Leftrightarrow \;\; &x=y+\delta^{n-1}_+(w)-\delta^{n-1}_-(w)\quad\text{for some}\quad w \quad\text{in}\quad U(X^{n-1}).
\end{align*} \end{paragr}
\vskip+1.5mm
Now let us recall the definition of the cohomology monoids introduced in \cite{P6}.
\vskip+1.5mm
Let $M$ be a monoid and $A$ be a (left) $M$-semimodule. Define
$$F^n(M,A)=\{f \colon M^n \lr A \; | \; f(x_1, \cdots , x_{j-1}, 1, x_{j+1}, \cdots , x_n)=0,\; j=1,2, \cdots, n\}, \; \; \; n \geq 0.$$
Clearly, $F^n(M,A)$, together with the usual addition of functions, is an abelian monoid. Next, define monoid homomorphisms $d^n_{-},\; d^n_{+} \colon F^n(M,A) \longrightarrow F^{n+1}(M,A)$ as follows:
$$(d^n_{\pm}f)(x_1, \cdots, x_{n+1})=0 , \; n \geq 0,\; \text{if any}\; x_j=1, \; j=1,2, \cdots, (n+1).$$
If each $x_j \neq 1, j=1,2, \cdots, (n+1)$, then
$$(d_{+}^{2k}f)(x_1, \cdots, x_{2k+1})=\sum_{i=1}^{k}f(x_1, \cdots, x_{2i-1}x_{2i}, \cdots, x_{2k+1}) + f(x_1, \cdots, x_{2k}), \;\; k \geq 0,$$
$$(d_{-}^{2k}f)(x_1, \cdots, x_{2k+1})= x_1f(x_2,\cdots , x_{2k+1})+ \sum_{i=1}^{k}f(x_1, \cdots, x_{2i}x_{2i+1}, \cdots, x_{2k+1}), \;\; k \geq 0,$$
$$(d_{+}^{2k-1}f)(x_1, \cdots, x_{2k})= x_1f(x_2,\cdots , x_{2k})+ $$ $$
+\sum_{i=1}^{k-1}f(x_1, \cdots, x_{2i}x_{2i+1}, \cdots, x_{2k}) + f(x_1, \cdots, x_{2k-1}), \;\; k \geq 1,$$
$$(d_{-}^{2k-1}f)(x_1, \cdots, x_{2k})= \sum_{i=1}^{k}f(x_1, \cdots, x_{2i-1}x_{2i}, \cdots, x_{2k}), \;\; k \geq 1.$$
\vskip+3mm
It is easy to see that
$$d_+^nd_+^{n-1} + d_-^nd_-^{n-1}=d_+^nd_-^{n-1}+d_-^nd_+^{n-1}$$
for all $n \geq 1$. In other words, the sequence
$$\xymatrix{ F(M,A) \colon \;\;\; 0 \ar@<0.6ex>[r] \ar@<-0.6ex>[r] & F^{0}(M,A) \ar@<0.6ex>[r]^-{d^{0}_{+}} \ar@<-0.6ex>[r]_-{d^{0}_{-}} & F^1(M,A) \ar@<0.6ex>[r] \ar@<-0.6ex>[r]  & \cdots \ar@<0.6ex>[r] \ar@<-0.6ex>[r] & F^{n}(M,A) \ar@<0.6ex>[r]^{d^{n}_{+}} \ar@<-0.6ex>[r]_-{d^{n}_{-}} & F^{n+1}(M,A) \ar@<0.6ex>[r] \ar@<-0.6ex>[r] & \cdots }$$
is a nonnegative cochain complex of abelian monoids. The {\em (normalized) $n$-th cohomology monoid $H^n(M, A)$ of $M$ with coefficients in the $M$-semimodule $A$} is defined by
$$H^n(M, A) = H^n(F(M,A)), \; \; n \geq 0.$$

It is obvious that any homomorphism $\alpha \colon A\lr B$ of $M$-semimodules induces a $\pm$-morphism $F(M,\alpha): F(M,A) \lr F(M,B), \;                      F^n(M,\alpha)(f\colon M^n \lr A)=\alpha f$. Consequently, $H^n(M, -)$ is a covariant additive functor from the category of $M$-semimodules to the category of abelian monoids (see \ref{ipl}).

\begin{paragr}\label{bip} As special cases, we have
$$H^0(M,A) = \{a \in A\; |\; xa=a \;\;\text{for all} \; x \in M\};$$

$$H^1(M,A) = \{f \colon M \lr A\; |\; f(1)=0 \; \text{and}\; xf(y)+f(x)=f(xy)\; \;\text{for all}\; x,y \in M \}/ \rho^1,$$
$$f \rho^1 f' \Leftrightarrow \exists \; a_1, a_2 \in A \; : \;  f(x)+xa_1+a_2=f'(x)+xa_2+a_1 \; \;\text{for all} \; x \in M;$$

\begin{align*} \begin{split}H^2(M,A) = \{f \colon M \times M \lr A\; |\; f(x, 1)=0=f(1,y) \; \text{and} \\ \; xf(y, z)+f(x, yz)=f(xy, z)+f(x, y)\; \;\text{for all}\; x,y,z \in M \}/ \rho^2, \end{split} \end{align*}
\begin{align*}f \rho^2 f' \Leftrightarrow \exists \; g_1, g_2\colon M \lr A\; :\; g_1(1)=0=g_2(1)\;\text{and}\; f(x,y)+xg_1(y)+g_1(x)+g_2(xy)=\\ =f'(x,y)+xg_2(y)+g_2(x)+g_1(xy) \; \;\text{for all} \; x,y \in M.  \hspace{4.23cm}\end{align*}\end{paragr}

\begin{paragr}\label{bipl} The cohomology monoids introduced in  \cite{P2} and \cite{HLS} (see the introduction of this paper) are defined as follows. Let $M$ be a monoid and $A$ an $M$-semimodule. Define
$$\mathscr{H}^n(M,A)=\big\{f\in F^n(M,A)\;|\;d^n_+(f)=d^n_-(f)\big\}/\widetilde{\rho}\;^n, \;\;\;\;n \geq 0,$$
where $\widetilde{\rho}\;^n$ is a congruence given by
\begin{align*}
f \;\widetilde{\rho}\;^n f'\;\;\Leftrightarrow \;\; &f=f'+d^{n-1}_+(g)-d^{n-1}_-(g)\quad\text{for some}\quad g \quad\text{in}\quad U(F^{n-1}(M,A))=F^{n-1}(M,U(A))
\end{align*} (see \ref{ipliki}). In particular, if  $f$ and $f'$ are $2$-cocycles, then $f \;\widetilde{\rho}\;^2 f'$  if and only if there is a function $g\colon M \lr U(A)$ with $g(1)=0$ such that
$$f(x,y)=f'(x,y)+xg(y)-g(xy)+g(x)$$
for all $x,y\in M$. \end{paragr}

\begin{paragr}\label{bipli} Let $A$ be an $M$-module, i.e., a module over $K(\N[M])= \Z[M]$, the integral monoid ring of the monoid $M$.\;Then the cohomology monoids $H^n(M, A)$ and $\mathscr{H}^n(M,A)$ both coincide with the $n$-th cohomology group of the nonnegative ordinary cochain complex $\{F^n(M,A), \;d_+^n-d_-^n\}$ of abelian groups (see \ref{ipli}), that is, with the $n$-th  Eilenberg-Mac Lane cohomology group of $M$ with coefficients in the $M$-module $A$. \end{paragr}

\begin{paragr}\label{biplik} If $A$ is a cancellative $M$-semimodule, then $H^n(M,A)$ is a cancellative abelian monoid for each $n$. Moreover, for any cancellative $M$-semimodule $A$, the homomorphism $H^n(M,k_A):H^n(M,A)\lr H^n(M,K(A))$ induced by the inclusion $ k_A : A \hookrightarrow K(A)$ is injective (cf. \ref{iplik}). \end{paragr}

For any $M$-semimodule $A$, one has a commutative diagram of abelian monoids and monoid homomorphisms
\begin{align}\xymatrix{\mathscr{H}^n(M,A)\ar[rr]^{k_A^n}\ar[dr]_{j^n} & & H^n(M,K(A))  \\ &  \ar[ru]_{\;\;\;\;H^n(M,k_A)} H^n(M,A),}\end{align}
where $H^n(M,k_A)$, \;$j^n$ and \;$k_A^n$ are defined by
$$ H^n(M,k_A)(cl(f))=cl(k_Af),\;\;\;\;j^n(cl(f))=cl(f),\;\;\;\;k_A^n(cl(f))=cl(k_Af).$$
Clearly,  $j^n$ is a surjection. If $A$ is a cancellative $M$-semimodule, then  $H^n(M,k_A)$ is an injection by \ref{biplik}, and if $A$ is an $M$-module, then all three maps are identity homomorphisms by \ref{bipli}.
\vskip+1.5mm
The cohomology monoids $\mathscr{H}^n(M,A)$ are adequate for many applications (see e.g. \cite{A,H1,H2,HLS,P1,P2}), but difficult to compute in general (see Remark 2.7 of \cite{P6}). Unlike $\mathscr{H}^n(M,A)$, the cohomology monoids  $H^n(M,A)$ can be introduced, as shown in \cite{P6}, by using an $M$-semimodule analog of the classical normalized bar resolution. This enables one to use the technique of free resolutions when calculating  $H^n(M,A)$, and therefore $H^n(M,A)$ is more computable alternative to $\mathscr{H}^n(M,A)$. In particular, we have

\begin{theo}[\cite{P6}] \label{bipliki} Let $C_m(t)$ be the multiplicative cyclic group of order $m$ on generator $t.$ If $A$ is a\; cancellative $C_m(t)$-semimodule, then

$$H^0(C_m(t),A)\cong\{a \in A \;\vert \; ta=a\},$$
\;
$$H^{2k}(C_m(t),A)\cong H^{2k}(C_m(t),K(A)), \;\;k > 0,$$
\;
$$H^{2k+1}(C_m(t),A)\cong\{a \in U(A) \;\vert \; (1+t+\cdots + t^{m-1})a=0\}/(U(A)\cap(t-1)K(A)), \;\; k \geq 0.$$

\end{theo}
Below, in Section 5, we need to know that there exists a cancellative $\Pi$-semimodule $A$, with $\Pi$ a group, such that $H^3(\Pi,A)=0$ while $H^3(\Pi,U(A))$ and $H^3(\Pi,K(A))$ both do not vanish. Therefore we conclude this section with

\begin{exmp} \label{imri} Let $m$ be a positive integer greater than 1, $\Z/m\Z$ the additive group of integers modulo $m$, $\N$ the additive monoid of non-negative integers, and let $D$ be a cancellative abelian monoid with $U(D)=0$ and satisfying the following condition: there are $d_1, d_2 \in D$ such that $d_1\neq d_2$ and $md_1=md_2$. (Example of $D$: $(\N\oplus\Z/m\Z) \backslash\{(0,1+m\Z), (0,2+m\Z),\cdots,(0,(m-1)+m\Z)\},  m(1,m\Z)=m(1,1+m\Z).)$ Define an action of $C_m(t)$ on the cancellative abelian monoid $A=D\oplus\N\oplus\Z/m\Z$ as follows:
$$t^i(d,\; n,\; p+m\Z)=(d,\;n,\;in+p+m\Z), \;\;\;d\in D,\;\; n\in \N,\;\; i,p \in\Z.$$
If $t^i=t^j$ and $p+m\Z=q+m\Z$, i.e., $i-j=rm$ and $p-q=sm$, $r, s \in \Z$, then $jn+q+m\Z=(i-rm)n+p-sm+m\Z=in+p-(rn+s)m+m\Z=in+p+m\Z$. Hence, the action is well defined. Furthermore, $A$ is a $C_m(t)$-semimodule under this action. The induced actions on $U(A)=\Z/m\Z$ and $K(A)=K(D)\oplus\Z\oplus\Z/m\Z$ are obvious: $C_m(t)$ acts on $U(A)$ trivially, and on $K(A)$ by $t^i(c,\; z,\; p+m\Z)=(c,\;z,\;iz+p+m\Z)$. It is easy to check that $\{b\in K(A)\;\vert\;(1+t+\cdots+t^{m-1})b=0\}=\{c\in K(D)\;\vert\;mc=0\}\oplus\Z/m\Z$ and  $(t-1)K(A)=\Z/m\Z$. Consequently,
$$H^{2k+1}(C_m(t),U(A))=H^{2k+1}(C_m(t),\Z/m\Z)=\Z/m\Z, \;\;\;\;k\geq0,$$
$$H^{2k+1}(C_m(t),K(A))=\{b\in K(A)\;\vert\;(1+t+\cdots+t^{m-1})b=0\}/(t-1)K(A)=$$
$$=\{c\in K(D)\;\vert\; mc=0\}, \;\;\;k\geq0,$$
and, by \ref{bipliki},
$$H^{2k+1}(C_m(t),A)=\{a\in U(A)\;\vert \;ma=0\}/(\Z/m\Z)=0, \;\;\;\; k\geq0.$$
Clearly, $H^{2k+1}(C_m(t),K(A))\neq0$ since $d_1-d_2\in K(D), \; d_1-d_2\neq0$ and $m(d_1-d_2)=0.$ Thus,
$$H^{2k+1}(C_m(t),U(A))\neq0, \;\; \;H^{2k+1}(C_m(t),K(A))\neq0, \quad\text{and}\quad  H^{2k+1}(C_m(t),A)=0,  \;\; k\geq0.$$
\end{exmp}

\vspace{1cm}
\section{Schreier extensions of semimodules by monoids}
Let $M$ be a monoid and $A$ a (left) $M$-semimodule. In this section we first give a cohomological classification of Schreier extensions of $A$ by $M$ via
$\mathscr{H}^2(M,A)$, which is part of the author's unpublished  thesis \cite{PhD}. Then, in the rest of the section, a relationship between $H^2(M, A)$
and Schreier extensions of $A$ by $M$ is described.
\begin{defi}[\cite{R,I,Str}]\label{mam} A sequence
$\xymatrix{E: A \ar@{ >-{>}}[r]^-\kappa & B \ar@{ ->>}[r]^-\sigma & M}$ of (not necessarily abelian) monoids and
monoid homomorphisms is called a (right) Schreier extension of $A$ by $M$ (some authors would say ``$M$ by $A$'') if the following conditions hold:

 \begin{enumerate}
\item \ $\kappa$ is injective, $\sigma$ is surjective, and $\kappa(A)=\Ker(\sigma)\;\;\; (\Ker(\sigma)=\{b\in B|\sigma(b)=1\})$.
\item \  For any $x \in M$, $\sigma^{-1}(x)$ contains an element $u_x$ such that for any
$b \in \sigma^{-1}(x)$ there exists a unique element $a\in A$ with $b=\kappa(a)+u_x$ (the monoid composition in $A$ and $B$ is written as addition).
\end{enumerate}
The elements $u_x$, $x\in M$, are called representatives of the extension
 $\xymatrix{E: A \ar@{ >-{>}}[r]^-\kappa & B \ar@{ ->>}[r]^-\sigma & M}$. \end{defi}
It follows in particular that $\sigma$ is a cokernel of $\kappa$.
\begin{exmp}Let $\N$ be the additive monoid of non-negative integers and let $C_m(t)$, as above, denote the multiplicative cyclic group of order $m$ on generator $t.$ The sequence of abelian monoids
$$\xymatrix{\N \ar@{ >-{>}}[r]^-{\bar{m}} & \N \ar@{ ->>}[r]^-p & C_m(t)},\;\;\;\;\;\bar{m}(1)=m,\;\;\;\;\;p(1)=t,$$
is a Schreier extension of $\N$ by $C_m(t)$. The set of representatives consists of $0, 1, ..., m-1.$
\end{exmp}
\begin{paragr} \label{mami} Suppose given a Schreier extension of monoids $\xymatrix{E: A \ar@{ >-{>}}[r]^-\kappa & B \ar@{ ->>}[r]^-\sigma & M}$ and a
representative $u_x \in B$. It follows from the above definition that an element $b\in \sigma^{-1}(x)$ serves as a representative of $E$ if and only if there is $a\in U(A)$ such that $b=\kappa(a)+u_x$. In particular, if $\xymatrix{E:A \ar@{ >-{>}}[r]^-\kappa & B \ar@{ ->>}[r]^-\sigma & M}$ is a Schreier extension with $A$ a group, then any $b\in B$ is a representative of the extension $E$.
\end{paragr}
\begin{paragr}\label{mamik} A morphism from $\xymatrix{E:A\;\ar@{>->}[r]^-\kappa&B\ar@{->>}[r]^-\sigma&M}$ to $\xymatrix{E':A'\;\ar@{>->}[r]^-{\kappa'}&B'\ar@{->>}[r]^{\sigma'}&M'}$ is a triple of monoid homomorphisms $(\al,\bt,\gm)$ such that
$$  \xymatrix{E:\hskip-13mm&A\;\ar[d]_\al\ar@{>->}[r]^-{\kappa}&
B\ar@{->>}[r]^\sigma\ar[d]_\bt&M\ar[d]_\gm\\
E':\hskip-12mm&A'\;\ar@{>->}[r]^-{\kappa'}&B'\ar@{->>}[r]^{\sigma'}&M'}$$
is a commutative diagram, and $\bt$ preserves representatives, that is, for any representative $u_x$ of $E$, $\beta(u_x)$ is a representative of $E'$. It follows from \ref{mami} that $\bt$ in the above commutative diagram preserves representatives if and only if for each $x\in M$ there exists a representative $u_x\in \sigma^{-1}(x)$ whose image under $\beta$ is a representative of $E'$.\end{paragr}

\begin{prop}\label{mamiko} Let $(\al,\bt,\gm)$ be a morphism from a Schreier extension $\xymatrix{E:A\;\ar@{>->}[r]^-\kappa&B\ar@{->>}[r]^-\sigma&M}$ to a Schreier extension $\xymatrix{E':A'\;\ar@{>->}[r]^-{\kappa'}&B'\ar@{->>}[r]^{\sigma'}&M'}$. Then:

{\rm $(a)$} If $\al$ and $\gm$ are injective homomorphisms, so is $\bt$.

{\rm $(b)$} If $\al$ and $\gm$ are surjective homomorphisms, so is $\bt$.

{\rm $(c)$} If $\al$ and $\gm$ are isomorphisms, so is $\bt$. \end{prop}

\begin{proof}
$(a)$ Suppose $\bt(b_1)=\bt(b_2).$ Then $\gm\sigma(b_1)=\sigma'\bt(b_1)=\sigma'\bt(b_2)=\gm\sigma(b_2)$, whence $\sigma(b_1)=\sigma(b_2).$
Consequently, $b_1=\kappa(a_1)+u_x$ and $b_2=\kappa(a_2)+u_x$, where $a_1, a_2 \in{A}, \;x=\sigma(b_1)$ and $u_x$ is a representative of the extension $E$. Next, $\kappa'\al(a_1)+\bt(u_x)=\bt\kappa(a_1)+\bt(u_x)=\bt(\kappa(a_1)+u_x)=\bt(b_1)=\bt(b_2)=\bt(\kappa(a_2)+u_x)=\bt\kappa(a_2)+\bt(u_x)=
\kappa'\al(a_2)+\bt(u_x).$ Since $\bt$ preserves representatives, it follows that $\al(a_1)=\al(a_2)$, whence $a_1=a_2.$ Thus $b_1=b_2$.

$(b)$ Let $b'\in{B'}.$  Take an element $x$ of $M$ with $\gm(x)=\sigma'(b')$ and a representative $u_x$ of the extension $E$. As $\sigma'(b')=\gm\sigma(u_x)=\sigma'\bt(u_x)$ and $\bt(u_x)$ is a representative of $E'$, there exists a (unique) element $a'$ of $A'$ with $b'=\kappa'(a')+\bt(u_x)$. Then, since $a'=\al(a)$ for some $a\in{A}$, one has $b'=\kappa'\al(a)+\bt(u_x)=\bt\kappa(a)+\bt(u_x)=\bt(\kappa(a)+u_x).$

Clearly $(a)$ and $(b)$ together yield $(c)$.
\;
\end{proof}

This proposition generalizes the Schreier split short five lemma of \cite{B} on the one hand, and the short five lemma for Schreier extensions of abelian monoids, proved in \cite{PhD}, on the other hand.
\vskip+1mm
 Two Schreier extensions $\xymatrix{E:A\;\ar@{>->}[r]&B\ar@{->>}[r]&M}$ and $\xymatrix{E':A\;\ar@{>->}[r]&B'\ar@{->>}[r]&M}$ are {\em congruent}, $E\equiv E'$, if there exists a monoid homomorphism $\bt:B\lr B'$ such that the diagram
 $$  \xymatrix{E:\hskip-13mm&A\;\ar@{=}[d]_{1_{{}_A}}\ar@{>->}[r]
&B\ar@{->>}[r]\ar[d]_\bt&M\ar@{=}[d]_{1_{{}_M}}\\
E':\hskip-12mm&A\;\ar@{>->}[r]&B'\ar@{->>}[r]&M} $$
commutes and $\bt$ carries representatives to representatives, i.e., $(1_A, \bt, 1_M)$ is a morphism from $E$ to $E'$.
 By Proposition \ref{mamiko}$(c)$, $\bt$ is an isomorphism. Consequently, congruence of extensions is a reflexive, symmetric and transitive relation.

\begin{defi}[\cite{P1}] \label{mamikos} Let $A$ be an $M$-semimodule and $\xymatrix{E: A \ar@{ >-{>}}[r]^-\kappa & B \ar@{ ->>}[r]^-\sigma & M}$  a sequence of monoids and monoid homomorphisms.  We say that $E$ is a {\em Schreier extension of the $M$-semimodule $A$ by the monoid $M$} if the following conditions hold:
\begin{enumerate}
\item\ $E$ is a Schreier extension of monoids.
\item\ $b+\kappa(a)=\kappa(\sigma(b)a)+b$ for all $a\in A$ and $b\in B.$
\end{enumerate}
\end{defi}

Two Schreier extensions of an $M$-semimodule $A$ by $M$ are congruent if they are congruent as Schreier extensions of monoids.
\vskip+1.5mm
For a given $M$-semimodule $A$, there is at least one Schreier extension of $A$ by $M$, the semidirect product extension
$$\xymatrix{0:A\;\ar@{>->}[r]^-\iota&A\rtimes M\ar@{->>}[r]^-\pi&M},\;\;\; \iota(a)=(a,1),\;\; \pi(a,x)=x$$
($(a,x)$ serves as a representative if and only if $a\in U(A)$).

Note that a Schreier extension $\xymatrix{E: A \ar@{ >-{>}}[r]^-\kappa & B \ar@{ ->>}[r]^-\sigma & M}$ is congruent to the semidirect product extension of $A$ by $M$ if and only if there is a monoid homomorphism  $\nu:M\lr B$ such that $\sigma\nu=1$ and $\nu(x)$ is a representative of $E$ for each $x\in M$ \cite[Proposition 5.4]{PhD}.

\begin{prop}[{\cite[Lemma 5.6]{PhD}}]\label{mamikoseb} Let $M$ be a monoid and $\xymatrix{E:A\;\ar@{>->}[r]^-\kappa&B\ar@{->>}[r]^-\sigma&M}$  a Schreier extension of an $M$-semimodule $A$ by $M$, and let $\al:A\lr A'$ be a homomorphism of $M$-semimodules. Then:

{\rm (i)} There exists a Schreier extension  $\xymatrix{\al E:A'\;\ar@{>->}[r]^-{\kappa'}&B'\ar@{->>}[r]^-{\sigma'}&M}$ and a morphism

$$  \xymatrix{E:\hskip-13mm&A\;\ar[d]_\al\ar@{>->}[r]^-{\kappa}&
B\ar@{->>}[r]^\sigma\ar[d]_\bt&M\ar@{=}[d]_{1_{{}_M}}\\
\al E:\hskip-12mm&A'\;\ar@{>->}[r]^-{\kappa'}&B'\ar@{->>}[r]^{\sigma'}&M}$$

of extensions.

{\rm (ii)} If $\xymatrix{E': A' \ar@{ >-{>}}[r]^-{\kappa''} & B'' \ar@{ ->>}[r]^-{\sigma''} & M}$ is another Schreier extension of $A'$ by $M$ with a morphism
$$  \xymatrix{E:\hskip-13mm&A\;\ar[d]_\al\ar@{>->}[r]^-{\kappa}&
B\ar@{->>}[r]^\sigma\ar[d]_{\bt'}&M\ar@{=}[d]_{1_{{}_M}}\\
 E':\hskip-12mm&A'\;\ar@{>->}[r]^-{\kappa''}&B''\ar@{->>}[r]^{\sigma''}&M}$$
of extensions, then $E'\equiv \al E.$
\end{prop}

\begin{proof}
{\rm (i)} To each $x$ in $M$ choose a representative $u_x$ of the extension $E$. In particular, choose $u_1=0$. Regard $A'$  as a $B$-semiodule
($ba'=\sigma(b)a' $) and consider the semidirect product $A'\rtimes B$. Define a relation $\rho$ on the monoid $A'\rtimes B$ as follows: If
$(a'_1,b_1),(a'_2,b_2)\in A'\rtimes B $ and $b_1=\kappa(a_1)+u_{\sigma(b_1)},\; b_2=\kappa(a_2)+u_{\sigma(b_2)},\; a_1,a_2 \in A,$ we let
$((a'_1,b_1),(a'_2,b_2))\in \rho$ if and only if $u_{\sigma(b_1)}=u_{\sigma(b_2)}$ and $a'_1+\al(a_1)=a'_2+\al(a_2).$ One can easily see that
the relation $\rho$ is in fact a congruence on the monoid $A'\rtimes B .$ Denote by $B'$ the quotient monoid $(A'\rtimes B)/\rho$ and define
monoid homomorphisms $\kappa':A'\lr B'$, $\sigma':B'\lr M$ and $\bt:B\lr B'$ by $\kappa'(a')=cl(a',0), \; \sigma'(cl(a',b))=\sigma(b)$ and $\bt(b)=
cl(0,b)$, respectively. It is straightforward to check that $$\xymatrix{\al E: A' \ar@{ >-{>}}[r]^-{\kappa'} & B' \ar@{ ->>}[r]^-{\sigma'}&M}$$
is a Schreier extension of $A'$ by $M$ (for each $x\in M$, $cl(0,u_x)$ serves as a representative of $\al E$) and $(\al,\bt,1_M)$ a morphism
from $E$ to $\al E$.

{\rm (ii)} Define $\bt'':B'\lr B''$ by $\bt''(cl(a',b))=\kappa''(a')+\bt'(b).$ Consider $(a'_1,b_1),(a'_2,b_2)\in A'\rtimes B $ with $b_1=\kappa(a_1)+u_{\sigma(b_1)},\; b_2=\kappa(a_2)+u_{\sigma(b_2)}$ and let $((a'_1,b_1),(a'_2,b_2))\in \rho$, i.e., $u_{\sigma(b_1)}=u_{\sigma(b_2)}$ and $a'_1+\al(a_1)=a'_2+\al(a_2).$ Then
$\kappa''(a'_1)+\bt'(b_1)=\kappa''(a'_1)+\bt'(\kappa(a_1)+u_{\sigma(b_1)})=\kappa''(a'_1+\al(a_1))+\bt'(u_{\sigma(b_1)})=\kappa''(a'_2+\al(a_2))+
\bt'(u_{\sigma(b_2)})=\kappa''(a'_2)+\bt'(\kappa(a_2)+u_{\sigma(b_2)})=\kappa''(a'_2)+\bt'(b_2).$ Thus $\bt''$ is well defined. Clearly, $\bt''$ carries 0 into 0.  Besides, in view of \ref{mamikos}(2), we can write $\bt''(cl(a'_1,b_1))+\bt''(cl(a'_2,b_2))=\kappa''(a'_1)+\bt'(b_1)+\kappa''(a'_2)+\bt'(b_2)=\kappa''(a'_1)+
\kappa''(\sigma''(\bt'(b_1))a'_2)+\bt'(b_1)+\bt'(b_2)=\kappa''(a'_1)+\kappa''(\sigma(b_1)a'_2)+\bt'(b_1)+\bt'(b_2)=\kappa''(a'_1+\sigma(b_1)a'_2)+\bt'(b_1+b_2)=
\bt''(cl(a'_1+\sigma(b_1)a'_2,\;b_1+b_2))=\bt''(cl(a'_1,b_1)+cl(a'_2,b_2)).$ Hence $\bt''$ is a homomorphism of monoids. The commutativity of the diagram

$$  \xymatrix{\al E:\hskip-13mm&A'\ar@{=}[d]_{1_{{}_{A'}}}\;\ar@{>->}[r]^-{\kappa'}&
B'\ar@{->>}[r]^{\sigma'}\ar[d]_{\bt''}&M\ar@{=}[d]_{1_{{}_M}}\\
E':\hskip-12mm&A'\;\ar@{>->}[r]^-{\kappa''}&B''\ar@{->>}[r]^{\sigma''}&M} $$
is obvious. Furthermore, $\bt''$ carries representatives to representatives. Indeed, it suffices to note that $\bt''(cl(0,u_x))=\bt'(u_x)$ is a
representative of $E'$ for each $x\in M$ (see \ref{mamik}). Hence $\al E\equiv E'.$
\end{proof}

The construction we described in the previous proposition is similar to a construction given in \cite{NAM}, Theorem 4.1. On the one hand, the construction in \cite{NAM} is given for a more restricted class of extensions, since the authors consider Schreier extensions whose kernel is an abelian group. On the other hand, they proved a stronger universal property, with respect to any homomorphism of $M$-modules, and not only with respect to the identity between the kernels.

\vskip+1mm

Proposition \ref{mamikoseb} gives the following implication and congruences
\begin{align} E_1\equiv E_2\Rightarrow \al E_1\equiv \al E_2,\end{align}
\begin{align} \al'(\al E)\equiv (\al'\al) E,\;\;\; 1_AE\equiv E.\end{align}
Furthermore, a morphism of semidirect product extensions
$$ \xymatrix{0:\hskip-13mm&A\;\ar[d]_\al\ar@{>->}[r]^-{\iota}&
A\rtimes M\ar@{->>}[r]^\pi\ar[d]_{\al\rtimes M}&M\ar@{=}[d]_{1_{{}_M}}\\
0:\hskip-12mm&A'\;\ar@{>->}[r]^-{\iota'}&A'\rtimes M\ar@{->>}[r]^{\pi'}&M}$$
implies, by Proposition \ref{mamikoseb}, that
\begin{align}\al0\equiv0.\end{align}

Let $\mathscr{E}(M,A)$ denote the set of all congruence classes of Schreier extensions of an $M$-semimodule $A$ by $M$. We regard $\mathscr{E}(M,A)$ as a pointed set with the distinguished element $cl(0)$. Moreover, (4.8), (4.9) and (4.10) show that $\mathscr{E}(M,-)$ can be regarded as a functor from the category of $M$-semimodules to the category of pointed sets (for a homomorphism $\al:A\lr A'$ of $M$-semimodules, $\mathscr{E}(M,\al):\mathscr{E}(M,A)\lr \mathscr{E}(M,A')$ is defined by $\mathscr{E}(M,\al)(cl(E))=cl(\al E)).$
\vskip+1.5mm

\begin{paragr} \label{mamikosebi} Let $M$ be a monoid and $\xymatrix{E:A\;\ar@{>->}[r]^-\kappa&B\ar@{->>}[r]^-\sigma&M}$  a Schreier extension of an $M$-semimodule $A$ by $M$. To each $x$ in $M$ choose a representative $u_x$ of the extension $E$. In particular, choose $u_1=0$. Clearly, for each pair $x,y \in M$, there is a unique element $f(x,y) \in A$ such that
$$u_x+u_y=\kappa(f(x,y))+u_{xy}.$$
By \ref{mamikos}(2), one has
$$u_x+(u_y+u_z)=\kappa(xf(y,z)+f(x,yz))+u_{xyz}.$$
On the other hand,
$$(u_x+u_y)+u_z=\kappa(f(x,y)+f(xy,z))+u_{xyz}.$$
Hence,
$$xf(y,z)+f(x,yz)=f(xy,z)+f(x,y)$$
for all $x,y,z\in M$. Besides, since $u_1=0$,
$$f(x,1)=0=f(1,y)$$
for all $x,y\in M$. Thus, the function $f \colon M \times M \lr A$ defined above is a $2$-cocycle (see \ref{bip}). It is called a {\em factor set} of the extension $E$.
\end{paragr}
\begin{paragr}\label{mamikose} \textbf{Convention}. Suppose that $f$ and $f'$ are $2$-cocycles of a monoid $M$ with coefficients in an  $M$-semimodule $A$. We will say that $f$ and $f'$ are cohomologous (resp. strongly cohomologous) if $f \rho^2 f'$ (resp. $f \;\widetilde{\rho}\;^2 f'$) (see \ref{bip} and \ref{bipl}). It is clear that strongly cohomologous $2$-cocycles are cohomologous. If $A$ is an $M$-module, then $2$-cocycles are cohomologous if and only if they are strongly cohomologous.\end{paragr}
\begin{lem}\label{pip} Let $M$ be a monoid and $\xymatrix{E:A\;\ar@{>->}[r]^-\kappa&B\ar@{->>}[r]^-\sigma&M}$  a Schreier extension of an $M$-semimodule $A$ by $M$. Any two factor sets of $E$ are strongly cohomologous.\end{lem}
\begin{proof} Suppose that $u_x$ and $u'_x$,\; $x\in M$, are two choices of representatives of $E$ with $u_1=u'_1=0$. Denote the corresponding factor sets by $f$ and $f'$, respectively. It follows from \ref{mami} that there is a function $g \colon M \lr U(A)$ such that $u_x=\kappa(g(x))+u'_x$ for all $x\in M$. Clearly, $g(1)=0$. By \ref{mamikos}(2), we have
$$u_x+u_y=\kappa(g(x))+u'_x+\kappa(g(y))+u'_y=\kappa(g(x)+xg(y)+f'(x,y))+u'_{xy}.$$
On the other hand,
$$u_x+u_y=\kappa(f(x,y)+g(xy))+u'_{xy}.$$
Hence,
$$f(x,y)=f'(x,y)+xg(y)-g(xy)+g(x)$$
for all $x,y\in M$. Thus $f$ and $f'$ are strongly cohomologous. \end{proof}
\begin{paragr} \label{pipk} Since a congruence of Schreier extensions maps representatives to representatives, it is obvious that congruent  Schreier extensions have the same factor sets.\end{paragr}
\begin{paragr} \label{pipka} Suppose that $f \colon M \times M \lr A$ is a $2$-cocycle, that is,
$$xf(y,z)+f(x,yz)=f(xy,z)+f(x,y) \quad\text{and}\quad f(x,1)=0=f(1,y)$$
for all $x,y,z\in M$. Then the set $B_f=A\times M$ with composition defined by
$$(a_1,x)+(a_2,y)=(a_1+xa_2+f(x,y),xy)$$
is an additive monoid with the idendity element $(0,1)$. Furthermore, it is easy to see that
$$\xymatrix{E_f:A\;\ar@{>->}[r]^-{\kappa_f}&B_f\ar@{->>}[r]^-{\sigma_f}&M}, \;\;\; \kappa_f(a)=(a,1), \;\;\; \sigma_f(a,x)=x,$$
is a Schreier extension of the $M$-semimodule $A$ by the monoid $M$ with the set of representatives $\{(a,x)\;|\;a \in U(A),\;x\in M\}$. In addition, since $(0,x)+(0,y)=(f(x,y),1)+(0,xy)$, the $2$-cocycle $f$ is one of the factor sets of $E_f$. Obviously, $E_0\equiv0$.
\end{paragr}
\begin{lem}\label{bez} Let $f,f' \colon M \times M \lr A$ be $2$-cocycles. If $f$ and $f'$ are strongly cohomologous, then $E_f$ and $E_{f'}$ are congruent.
\end{lem}
\begin{proof} By the hypothesis, there exists a function $g \colon M \lr U(A)$ with $g(1)=0$ such that
$$f(x,y)=f'(x,y)+xg(y)-g(xy)+g(x)$$
for all $x,y\in M$. Consider a diagram
$$  \xymatrix{E_f:\hskip-13mm&A\;\ar@{=}[d]_{1_{{}_A}}\ar@{>->}[r]^-{\kappa_f}&
B_ f\ar@{->>}[r]^{\sigma_f}\ar[d]_\bt&M\ar@{=}[d]_{1_{{}_M}}\\
E_{f'}:\hskip-12mm&A\;\ar@{>->}[r]^-{\kappa_{f'}}&B_{f'}\ar@{->>}[r]^{\sigma_{f'}}&M}$$
in which $\beta$ is defined by $\beta(a,x)=(a+g(x),x)$. One can easily check that $\beta$ is a monoid homomorphism and the diagram commutes. Furthermore, $\beta$ preserves representatives. Indeed, by \ref{mami} and \ref{mamik}, we only need to note that $\beta(0,x)=(g(x),x)=(g(x),1)+(0,x)$ and $g(x)\in U(A)$. Thus, $(1_A, \beta, 1_M)$ is a morphism of extensions. Hence $E_f \equiv E_{f'}$.
\end{proof}
\begin{paragr}\label{bezb} The converse of Lemma \ref{bez} also holds. Indeed, if $E_f \equiv E_{f'}$, then $f$ and $f'$ are both factor sets of $E_f$ (see \ref{pipka} and \ref{pipk}). Hence, by Lemma \ref{pip}, they are strongly cohomologous.
\end{paragr}
\begin{paragr}\label{bezbu} Suppose that $M$ is a monoid and $\xymatrix{E:A\;\ar@{>->}[r]^-\kappa&B\ar@{->>}[r]^-\sigma&M}$  a Schreier extension of an $M$-semimodule $A$ by $M$. Suppose further that $u_x,\;x\in M$, is a set of representatives of $E$ with $u_1=0$  and  $f \colon M \times M \lr A$ the corresponding factor set. Then
$$  \xymatrix{E_f:\hskip-13mm&A\;\ar@{=}[d]_{1_{{}_A}}\ar@{>->}[r]^-{\kappa_f}&
B_ f\ar@{->>}[r]^{\sigma_f}\ar[d]_\bt&M\ar@{=}[d]_{1_{{}_M}}\\
E:\hskip-12mm&A\;\ar@{>->}[r]^-{\kappa}&B\ar@{->>}[r]^{\sigma}&M,}$$
where $\beta$ is defined by $\beta(a,x)=\kappa(a)+u_x$, is a morphism of extensions. Hence $E_f$ is congruent to $E$.
\end{paragr}
\begin{theo}[\cite{P1,PhD}]\label{bix} Let $M$ be a monoid. For any $M$-semimodule $A$, the map
$$ \zeta(M,A):\mathscr{H}^2(M,A)\lr \mathscr{E}(M,A),\;\;\;\;\;\;\;\zeta(M,A)(cl(f))=cl(E_f),$$
is a bijection of pointed sets.  \end{theo}
\begin{proof} By Lemma \ref{bez}, $\zeta(M,A)$ is well-defined.  Since $E_0\equiv0$, we have $\zeta(M,A)(cl(0))=cl(E_0)=cl(0)$. The map
$$ \eta(M,A):\mathscr{E}(M,A)\lr \mathscr{H}^2(M,A),\;\;\;\;\;\;\;\eta(M,A)(cl(E))=cl(f_E),$$
where $f_E$ denotes one of the factor sets of $E$ (see \ref{mamikosebi}), is also well-defined by Lemma \ref{pip} and \ref{pipk}. It follows from \ref{bezbu} that $\zeta\eta=1$. Besides, $\eta\zeta=1$ (see \ref{pipka}).\end{proof}
If $M$ is a group and $A$ an $M$-module, then $\zeta(M,A)$ turns into the well-known bijection between the second cohomology group of M with coefficients in $A$ and the set of all congruence classes of extensions of $A$ by $M$. In the case where $M$ is a monoid and $A$ an $M$-module, this theorem is obtained in \cite{T}.
\begin{remk} Any Schreier extension of monoids $\xymatrix{E: A \ar@{ >-{>}}[r]^-\kappa & B \ar@{ ->>}[r]^-\sigma & M}$ with $A$ a cancellative abelian monoid induces an action $\varphi:M\lr \text{End}(A)$ of $M$ on $A$ such that $b+\kappa(a)=\kappa(\varphi(\sigma(b))(a))+b$ \;for all $a\in A$ and $b\in B$\;(cf.\;\ref{mamikos}). Indeed, whenever $\kappa(a_1)+b=\kappa(a_2)+b,\; a_1, a_2 \in A,\; b\in B,$ one has $a_1=a_2.\;(\kappa(a_1)+b=\kappa(a_2)+b \Rightarrow \kappa(a_1+a)+u_{\sigma(b)}=
\kappa(a_2+a)+u_{\sigma(b)} \Rightarrow a_1+a=a_2+a \Rightarrow a_1=a_2.)$ Therefore, for each pair of elements $a\in A$ and $b\in B,$ there exists a unique element
$a'\in A$ such that $b+\kappa(a)=\kappa(a')+b.\;(b+\kappa(a)=\kappa(a_1)+u_{\sigma(b)}+\kappa(a)=\kappa(a_1)+\kappa(a')+u_{\sigma(b)}=
\kappa(a')+\kappa(a_1)+u_{\sigma(b)}=\kappa(a')+b.)$ Consequently, each $b\in B$ determines a map $\theta_b:A\lr A$ with $b+\kappa(a)=\kappa(\theta_b(a))+b,
\; a\in A,$ which is in fact an endomorphism of $A$. Moreover, the map $\theta:B\lr \text{End}(A)$ defined by $\theta(b)=\theta_b$ is a monoid homomorphism satisfying
$\theta\kappa=1_A.$ It hence determines the desired homomorphism $\varphi:M\lr \text{End}(A)$ with $\theta=\varphi\sigma$ (see remark after Definition \ref{mam}).
Thus, Theorem \ref{bix}, in particular, classifies Schreier extensions of a cancellative abelian monoid $A$ by a monoid $M$ which induce a fixed action
$\varphi:M\lr \text{End}(A)$ of $M$ on $A$.\end{remk}

Now, for an $M$-semimodule $A$, we examine the relationship between $H^2(M, A)$ and Schreier extensions of $A$ by $M$.

Let $M$ be a monoid and let $\xymatrix{E_1:A\;\ar@{>->}[r]^-{\kappa_1}&B_1\ar@{->>}[r]^-{\sigma_1}&M}$ and $\xymatrix{E_2:A\;\ar@{>->}[r]^-{\kappa_2}&B_2\ar@{->>}[r]^-{\sigma_2}&M}$ be  Schreier extensions of an $M$-semimodule $A$ by $M$. We say that $E_1$ and $E_2$ are {\em similar}, $E_1\sim E_2$, if there is a Schreier extension $\xymatrix{S: K(A)\ar@{ >-{>}}[r]^-\kappa & B \ar@{ ->>}[r]^-\sigma & M}$ of the $M$-module $K(A)$ by $M$ and monoid homomorphisms $\bt_1:B_1\lr B$ and $\bt_2:B_2\lr B$ such that the diagram

$$\xymatrix{\hskip-7mm E_1:A\; \ar@{>->}[r]^-{\kappa_1}\ar[d]_{k_A} & B_1\ar@{->>}[r]^-{\sigma_1}\ar[d]_{\bt_1} &M  \ar@{=}[d]_{1_{{}_M}} \\ \hskip-2mm S: K(A)\; \ar@{>->}[r]^-\kappa & B \ar@{->>}[r]^-\sigma & M \\ \hskip-7mmE_2: A \; \ar@{>->}[r]^-{\kappa_2} \ar[u]^{k_A} & B_2 \ar@{->>}[r]^-{\sigma_2} \ar[u]^{\bt_2} & M \ar@{=}[u]^{1_{{}_M}}}$$ is commutative. In other words, $E_1\sim E_2$ if there is a Schreier extension $\xymatrix{S: K(A)\ar@{ >-{>}}[r]^-\kappa & B \ar@{ ->>}[r]^-\sigma & M}$  and morphisms of extensions $(k_A, \bt_1, 1_M): E_1\lr S$ and $(k_A, \bt_2, 1_M): E_2\lr S$ (see \ref{mami} and \ref{mamik}).

Note that if $A$ is a cancellative $M$-semimodule, then  $\bt_1$ and $\bt_2$ in the above commutative diagram are in fact injective homomorphisms
by Proposition \ref{mamiko}$(a)$.
\begin{prop}\label{goyo} $E_1\sim E_2$ if and only if $k_AE_1\equiv k_AE_2.$  \end{prop}
\begin{proof} Suppose that $E_1\sim E_2.$ Then, by Proposition \ref{mamikoseb}, $S\equiv k_AE_1$ and $S\equiv k_AE_2.$ Hence $k_AE_1\equiv k_AE_2.$ Conversely, assume that $k_AE_1\equiv k_AE_2.$   That is, we are given a morphism $(1_{K(A)}, \bt, 1_M): k_AE_1\lr k_AE_2$ of extensions. On the other hand, again  by Proposition \ref{mamikoseb}, there are morphisms $(k_A, \bt_1, 1_M): E_1\lr k_AE_1$ and $(k_A, \bt_2, 1_M): E_2\lr k_AE_2.$ The latter and the composite morphism $(k_A, \bt\bt_1, 1_M): E_1\lr k_AE_2$ show that $E_1$ is similar to $E_2$.
\end{proof}
It immediately follows from this proposition that
\begin{align} E\sim E,\;\;\;\;\; E_1\sim E_2\Rightarrow E_2\sim E_1,\;\;\;\;\; E_1\sim E_2 \sim E_3 \Rightarrow E_1 \sim E_3, \end{align}
\begin{align} E_1 \equiv E_2 \Rightarrow E_1\sim E_2 \end{align}
(see (4.8)). Furthermore, (4.23) together with (4.9) and (4.10) implies the following similarities
\begin{align} \al'(\al E)\sim(\al'\al) E,\;\;\;\;\; 1_AE\sim E, \;\;\;\;\; \al0\sim0. \end{align}
Also, for any homomorphism of $M$-semimodules $\al:A\lr A'$, we have
\begin{align} E_1 \sim E_2 \Rightarrow \al E_1\sim \al E_2. \end{align}
Indeed, in view of Proposition \ref{goyo} and (4.8), there is a chain of implications
$$E_1\sim E_2\Rightarrow k_AE_1\equiv k_AE_2\Rightarrow K(\al)k_AE_1\equiv K(\al)k_AE_2\Rightarrow k_{A'}\al E_1\equiv k_{A'}\al E_2 \Rightarrow  \al E_1\sim \al E_2.$$
\vskip+1.5mm
Let $E(M,A)$ denote the set of all similarity classes of Schreier extensions of an $M$-semimodule $A$ by $M$ (see (4.22)) regarded as a pointed set with the distinguished element $cl(0)$.  By (4.24) and (4.25), $E(M,-)$ is a functor from the category of $M$-semimodules to the category of pointed sets (if $\al:A\lr A'$ is a homomorphism of $M$-semimodules, then $E(M,\al):E(M,A)\lr E(M,A')$  is defined by $E(M,\al)(cl(E))=cl(\al E)).$
\vskip+1mm
If $A$ is an $M$-module, then $E_1\sim E_2$ if and only if $E_1 \equiv E_2$, and hence $E(M,A)$ coincides with $\mathscr{E}(M,A)$.
\vskip+1.5mm
For any  $M$-semimodule $A$, Proposition \ref{goyo}, together with (4.8) and (4.23), yields a commutative diagram of pointed sets
\begin{align}\xymatrix{\mathscr{E}(M,A) \ar[rr]^{k_A^{*}} \ar[dr]_{j} & & E(M,K(A))  \\ &  \ar[ru]_{\;\;E(M,k_A)} E(M,A),}\end{align}
where $E(M,k_A)$, $j$ and $k_A^*$ are defined by
$$ E(M,k_A)(cl(E))=cl(k_AE),\;\;\;\;j(cl(E))=cl(E),\;\;\;\;k_A^*(cl(E))=cl(k_AE).$$  Clearly, $E(M,k_A)$ is an injection (by \ref{goyo}) and $j$ a surjection.

Before we continue, we note the following. Let $A$ be an $M$-semimodule, $f \colon M \times M \lr A$ a 2-cocycle and $\al:A\lr A'$ a homomorphism of $M$-semimodules. Then
\begin{align} \al E_f\equiv E_{\al f}, \end{align}
where $E_f$ and $E_{\al f}$ are the Schreier extensions corresponding to the 2-cocycles $f \colon M \times M \lr A$ and $\al f \colon M \times M \lr A'$ (see \ref{pipka}). Indeed, there is a morphism of extensions
$$  \xymatrix{E_f:\hskip-13mm&A\;\ar[d]_{\al}\ar@{>->}[r]^-{\kappa_f}&
B_ f\ar@{->>}[r]^{\sigma_f}\ar[d]_\bt&M\ar@{=}[d]_{1_{{}_M}}\\
E_{\al f}:\;\;\;\hskip-12mm&A'\;\ar@{>->}[r]^-{\kappa_{\al f}}&B_{\al f}\ar@{->>}[r]^{\sigma_{\al f}}&M,&\beta(a,x)=(\al(a),x),}$$
which implies the desired congruence by Proposition \ref{mamikoseb}. In particular, one has
\begin{align} k_A E_f\equiv E_{k_Af}. \end{align}
\begin{theo}\label{bixo} Let $M$ be a monoid. For any $M$-semimodule $A$ there is a surjective map of pointed sets
$$ \theta(M,A):H^2(M,A)\lr E(M,A),\;\;\;\;\;\;\;\theta(M,A)(cl(f))=cl(E_f).$$
Furthermore, if $A$ is a cancellative $M$-semimodule, then $\theta(M,A$) is a bijection. \end{theo}
\begin{proof} If two cocycles $f, f' \colon M \times M \lr A$ are cohomologous, then so are $k_Af$ and $k_Af'$, and therefore $E_{k_Af}$ is congruent to $E_{k_Af'}$ (see \ref{mamikose} and \ref{bez}). This congruence, in view of (4.28), yields a congruence $k_AE_f\equiv k_AE_{f'}$, whence, by Proposition \ref{goyo}, we conclude that $E_f$ and $E_{f'}$ are similar. Thus, $ \theta(M,A)$ is well-defined. If $E$ is a Schreier extension of $A$ by $M$ and $f$ is one of its factor sets (see \ref{mamikosebi}), then $E_{f}\equiv E$ (see \ref{bezbu}), and therefore $E\sim E_f$ by (4.23). Hence, $\theta(M,A)$ is a surjective map of pointed sets ($\theta(M,A)(cl(0))=cl(E_0)=cl(0)$ since $E_0\sim0$). Now suppose that $f_1, f_2 \colon M \times M \lr A$ are 2-cocycles such that $E_{f_1}\sim E_{f_2}$. This similarity, by Proposition \ref{goyo} and (4.28), implies that $E_{k_Af_1}\equiv E_{k_Af_2}$.  Hence, the  cocycles $k_Af_1 ,\; k_Af_2 \colon M \times M \lr K(A)$ are cohomologous (see \ref{mamikose} and \ref{bezb}). If $A$ is a cancellative $M$-semimodule, then the homomorphism $H^2(M,k_A):H^2(M,A)\lr H^2(M,K(A)), \;H^2(M,k_A)(cl(f))=cl(k_Af)$, is injective (see \ref{biplik}), and therefore the 2-cocycles $f_1, f_2 \colon M \times M \lr A$ are cohomologous as well.  Thus, $\theta(M,A)$ is a bijection of pointed sets for any cancellative $M$-semimodule $A$.\end{proof}

Note that the map $\theta(M,A)$ is functorial in $A$ (by (4.23) and (4.27)) and so is the map $\zeta(M,A)$ of Theorem \ref{bix} (by (4.27)).
\vskip+1.5mm
Theorems \ref{bix} and \ref{bixo}, together with the commutative diagrams (3.11) and (4.26), yield a commutative diagram

$$\xymatrix{H^2(M,A) \ar[rrr]^{\theta(M,A)}\ar[dd]_{H^2(M,k_A)} & & & E(M,A) \ar[dd]^{E(M,k_A)} \\ & \mathscr{H}^2(M,A)\ar[r]^{\zeta(M,A)}                    \ar[ul]_{j^2} \ar[dl]^{k_A^2}  & \mathscr{E}(M,A) \ar[ur]^{j} \ar[dr]_{k_A^*} \\ H^2(M,K(A)) \ar[rrr]^{\theta(M,K(A))} & & & E(M,K(A))}$$
for any $M$-semimodule $A$. Here $\zeta(M,A)$ and $\theta(M,K(A))$ are bijections, $\theta(M,A)$, $j$ and $j^2$ surjections, and $E(M,k_A)$ is an injection. If $A$ is a cancellative $M$-semimodule, then $\theta(M,A)$ is a bijection and $H^2(M,k_A)$ an injection, and if $A$ is an $M$-module, then $H^2(M,k_A),\; j^2,\; k_A^2,\; k_A^*,\; j$ and $E(M,k_A)$ are all identity maps.
\vskip+1.5mm

\section{An application to group extensions}
We end the paper with an application of $H^3(\Pi,A)$, where $\Pi$ is a group and $A$ a $\Pi$-semimodule, to a certain group extension problem.

Recall \cite{EM,M} that an abstract kernel is a triple $(\Pi, G, \psi)$ in which $\Pi$ and $G$ are groups and $\psi$ is a homomorphism from $\Pi$ to $\text{Aut}(G)/\text{In}(G)$. The homomorphism $\psi:\Pi\lr \text{Aut}(G)/\text{In}(G)$ induces a $\Pi$-module structure on the centre  $C$ of $G$ ($xc=\varphi(x)c, x\in\Pi, c\in C $ and $\varphi(x)\in \psi(x)$), and this $\Pi$-module is called the centre  of the abstract kernel $(\Pi, G, \psi)$. Any extension $\xymatrix{E:G\;\ar@{>->}[r]&B\ar@{->>}[r]&\Pi}$ of groups gives rise, by conjugation in $B$, to a homomorphism $\psi:\Pi\lr \text{Aut}(G)/\text{In}(G)$, called the abstract kernel of $E$. One says that an abstract kernel $(\Pi, G, \psi)$ has an extension if there exists an extension $\xymatrix{E:G\;\ar@{>->}[r]&B\ar@{->>}[r]&\Pi}$ whose abstract kernel coincides with $(\Pi, G, \psi)$.

In \cite{EM}, Eilenberg and Mac Lane showed that any abstract kernel $(\Pi, G, \psi)$ with centre $C$ determines  a well-defined $3$-dimensional cohomology class $\text{Obs}(\Pi, G, \psi)$ of $\Pi$ with coefficients in $C$, called the obstruction class of $(\Pi, G, \psi)$, and proved

\begin{theo}[{\cite[Theorem 8.1]{EM}}]\label{salome} An abstract kernel $(\Pi, G, \psi)$ has an extension if and only if $Obs(\Pi, G, \psi)=0$.\end{theo}

It follows, in particular, that if the third cohomology group $H^3(\Pi,C)$ of a group $\Pi$ with coefficients in a $\Pi$-module $C$ vanishes, then any abstract kernel $(\Pi, G, \psi)$ with centre $C$ has an extension.
\vskip+1.5mm
Now suppose that $C$ is the $\Pi$-module completion of a cancellative $\Pi$-semimodule $A$, i.e., $A$ is a $\Pi$-subsemimodule of the $\Pi$-module $C$ and $K(A)=C$. Example \ref{imri} shows that it may happen that $H^3(\Pi,A)=0$ while $H^3(\Pi,C)\neq0$. Hence, it may be useful to have the following immediate consequence of Theorem \ref{salome}.

\begin{prop} Let $A$ be a $\Pi$-subsemimodule of a $\Pi$-module $C$ such that $K(A)=C$. If $H^3(\Pi,A)=0$, then any abstract kernel $(\Pi, G, \psi)$ with centre $C$ whose obstruction class $Obs(\Pi, G, \psi)$ contains a $3$-cocycle of $\Pi$ with values in $A$ has an extension. \end{prop}

From now on, $S$ denotes an additively written (non-abelian) monoid.
\vskip+1.5mm
Every invertible element $u$ of a monoid $S$ induces, by conjugation with $u$, an inner automorphism $\mu_u$ of $S$,\; $\mu_u(s)=u+s-u,\; s\in S$. The set $\text{In}(S)$ of all inner automorphisms of $S$ forms a normal subgroup in the group $\text{Aut}(S)$, and hence gives rise to the quotient group $\text{Aut}(S)/\text{In}(S)$. If $\phi:\Pi\lr \text{Aut}(S)/\text{In}(S)$ is a homomorphism of groups, then the centre $A$ of $S$ can be regarded as a $\Pi$-semimodule, $xa=\varphi(x)a$ for any choice of automorphisms $\varphi(x)\in\phi(x)$, and this $\Pi$-semimodule is said to be the centre of the homomorphism $\phi$.

One says that a group $G$ is a group of left fractions of a monoid $S$ if $S$ is a submonoid of the group $G$ and every element of $G$ is expressible  in the form    $-s_1+ s_2$ with $s_1$ and $s_2$ in $S$. Recall that a monoid $S$ has a group of left fractions if and only if $S$ is a cancellative monoid and $(S+s)\cap(S+t)\neq\emptyset$ for all $s,t \in S$ (see e.g. \cite{CP}). A monoid satisfying these conditions is called an Ore monoid.

Let $S$ be an Ore monoid  and let $-S+S$ denote its group of left fractions. There is a group homomorphism
$$ \Phi_S:\text{Aut}(S)/\text{In}(S)\lr \text{Aut}(-S+S)/\text{In}(-S+S),\;\;\;\;\;\;\;\Phi_S(cl\varphi)=cl\wt{\varphi},$$
where $\wt{\varphi}$ is the extension of $\varphi \in \text{Aut}(S)$ to an automorphism of $-S+S, \;\wt{\varphi}(-s_1+s_2)=-{\varphi}(s_1)+
{\varphi}(s_2), \;s_1, s_2 \in S$. Consequently, every group homomorphism $\phi:\Pi\lr \text{Aut}(S)/\text{In}(S)$ with $S$ an Ore monoid  provides us with the abstract kernel $(\Pi,\; -S+S, \;\Phi_{S}\phi)$. In addition, the centre  of the homomorphism $\phi$ is a $\Pi$-subsemimodule of the centre of $(\Pi,\; -S+S, \; \Phi_{S}\phi)$.

\begin{prop}\label{niko} Let $A$ be the centre of a group homomorphism  $\phi:\Pi\lr \text{Aut}(S)/\text{In}(S)$ with $S$ an Ore monoid. Then $Obs(\Pi,\; -S+S, \; \Phi_{S}\phi)$ contains a $3$-cocycle with values in $U(A)$.\end{prop}
\begin{proof}In each automorphism class $\phi(x)$ select an automorphism $\varphi(x)$; in particular, choose $\varphi(1)=1$. Further, for each pair $x,y \in \Pi,$ select $f(x,y)\in U(S)$ such that $\varphi(x)\varphi(y)=\mu_{f(x,y)}\varphi(xy)$; in particular, choose $f(x,1)=0=f(1,y)$. Applying $\sim$, we obtain $\wt{\varphi(x)}\wt{\varphi(y)}=\wt{\mu_{f(x,y)}}\wt{\varphi(xy)}$, where $\wt{\mu_{f(x,y)}}$ is the inner automorphism of $-S+S$ induced by $f(x,y)$. Then, by \cite{EM,M}, there is a $3$-cocycle $k$ of $\Pi$ with coefficients in the centre of $(\Pi,\; -S+S, \; \Phi_{S}\phi)$ such that
$$\wt{\varphi(x)}(f(y,z))+f(x,yz)=k(x,y,z)+f(x,y)+f(xy,z)$$
for all $x,y,z \in \Pi$, and $\text{Obs}(\Pi,\; -S+S, \; \Phi_{S}\phi)=cl(k)$. Since $\wt{\varphi(x)}(f(y,z))=\varphi(x)(f(y,z))\in U(S)$, we have $k(x,y,z)\in U(S)$, whence we conclude that in fact the $3$-cocycle $k$ takes values in $U(A)$.\end{proof}

 This proposition together with Theorem \ref{salome} implies

\begin{cor}Let $A$ be a cancellative $\Pi$-semimodule. If $H^3(\Pi,A)=0$, then for any group homomorphism $\phi:\Pi\lr \text{Aut}(S)/\text{In}(S)$ with $S$ an Ore monoid  and with centre $A$, the abstract kernel $(\Pi,\; -S+S, \; \Phi_{S}\phi)$ has an extension.\end{cor}

This corollary is of some interest since it may happen that there is a cancellative $\Pi$-semimodule $A$, with $\Pi$ a group, such that $H^3(\Pi,A)=0$ while $H^3(\Pi,U(A))$ and $H^3(\Pi,K(A))$ both do not vanish (see Example \ref{imri}).

\vspace{1cm}

\begin{tabular}{l}
{\sc A.Razmadze Mathematical Institute,}\\
{\sc I. Javakhishvili Tbilisi State University}\\
{\sc Tamarashvili Str. 6, Tbilisi 0177, Georgia}\\

\vspace{1mm}

\\[-10pt]

\emph{E-mail address}: \texttt{alex.patchkoria@gmail.com}

\end{tabular}


\vspace{1cm}

















\begin{thebibliography}{XXXX}

\bibitem{A} A. Aljouiee, {\em On weak crossed products, Frobenius algebras, and the weak Bruhat
ordering}, J. Algebra, {\bf 287} (2005), 88-102.

\bibitem{B} D. Bourn, N. Martins-Ferreira, A. Montoli, M. Sobral, {\em Schreier split epimorphisms between
monoids}, Semigroup Forum {\bf 88} (2014), 739-752.

\bibitem{CP} A. Clifford and G. Preston, {\em The Algebraic Theory of Semigroups I}, Amer. Math. Soc., Mathematical Surveus 7, 1961.

\bibitem{EM} S. Eilenberg and S. Mac Lane, {\em Cohomology Theory in Abstract Groups. II. Group Extensions with a non-Abelian Kernel},
Ann. of Math., {\bf 48} (1947), 326-341.

\bibitem{Gol} J. S. Golan, {\em Semirings and Their Applications},
Kluwer Academic Publishers,  Dordrecht-Boston-London, 1999.

\bibitem{H1} D. Haile, {\em On crossed product algebras arising from weak cocycles}, J. Algebra, {\bf 74} (1982), 270-279.

\bibitem{H2} D. Haile, {\em The Brauer monoid of a field}, J. Algebra, {\bf 81} (1983), 521-539.

\bibitem{HLS} D. Haile, R. Larson and M. Sweedler, {\em A new invariant for $C$ over R:
Almost invertible cohomology theory and the classification of idempotent
cohomology classes and algebras by partially ordered sets with a Galois group
action}, Amer. J. Math., {\bf 105} (1983), 689-814.

\bibitem{I} H. Inassaridze, {\em Extensions of regular semigroups},
Bull. Georgian Acad. Sci., {\bf 39}(1965), 3-10 (in Russian).

\bibitem{M} S. Mac Lane, {\em Homology}, Springer-Verlag, Berlin-G\"ottingen-Heidelberg, 1963.

\bibitem{NAM} N. Martins-Ferreira, A. Montoli, M. Sobral, {\em The nine lemma and the push forward construction for special Schreier extensions of monoids}, DMUC {\bf 16-17} (2016) Preprint.

\bibitem{P1} A. Patchkoria, {\em Extensions of semimodules by monoids and their cohomological characterization}, Bull. Georgian
 Acad. Sci., {\bf 86} (1977), 21-24 (in Russian).

\bibitem{P2} A. Patchkoria, {\em Cohomology of monoids with coefficients in semimodules}, Bull. Georgian Acad. Sci., {\bf 86} (1977),
 545-548 (in Russian).

\bibitem{PhD} A. Patchkoria, {\em On Schreier extensions of semimodules}, PhD thesis, Tbilisi State University, 1979 (in Russian).

\bibitem{P3} A. Patchkoria, {\em Homology and cohomology monoids of presimplicial semimodules},
Bull. Georgian Acad. Sci., {\bf 162} (2000), 9-12.

\bibitem{P4} A. Patchkoria, {\em Chain complexes of cancellative semimodules},
Bull. Georgian Acad. Sci., {\bf 162} (2000), 206-208.

\bibitem{P5} A. Patchkoria,  {\em On exactness of long sequences of homology semimodules},  Journal of Homotopy and Related Structures, {\bf 1} (2006), 229-243.

\bibitem{P6} A. Patchkoria,  {\em Cohomology monoids of monoids with coefficients in semimodules I}, Journal of Homotopy and Related Structures, {\bf 9} (2014), 239-255.

\bibitem{R} L. Redei,  {\em Die Verallgemeinerung der Schreierschen Erweiterungstheorie},
Acta Sci. Math. Szeged, {\bf 14}(1952), 252-273.

\bibitem{Str} R. Strecker,  {\em \"{U}ber kommutative Schreiersche Halbgruppenerweiterungen},
Acta Math. Acad. Sci. Hungar., {\bf 23}(1972), 33-44.

\bibitem{S} M. Sweedler, {\em Multiplication alteration by two-cocycles}, Illinois J. Math., {\bf 15} (1971), 302-323.

\bibitem{T} N. X. Tuyen, {\em On cohomology and extensions of monoids}, PhD thesis, Tbilisi State University, 1977 (in Russian).

\end{thebibliography}
\end{document}